\documentclass[pdflatex,sn-mathphys-num]{sn-jnl}% Math and Physical Sciences Numbered Reference Style
\usepackage{graphicx}%
\usepackage{multirow}%
\usepackage{amsmath,amssymb,amsfonts}%
\usepackage{amsthm}%
\usepackage{mathrsfs}%
\usepackage[title]{appendix}%
\usepackage{xcolor}%
\usepackage{textcomp}%
\usepackage{manyfoot}%
\usepackage{booktabs}%
\usepackage{algorithm}%
\usepackage{algorithmicx}%
\usepackage{algpseudocode}%
\usepackage{listings}%
\theoremstyle{thmstyleone}%
\newtheorem{theorem}{Theorem}%  meant for continuous numbers
\newtheorem{proposition}[theorem]{Proposition}% 

\theoremstyle{thmstyletwo}%
\newtheorem{remark}{Remark}%

\theoremstyle{thmstylethree}%
\newtheorem{definition}{Definition}%

\newtheorem{lemma}{Lemma} % optional, if needed
\newtheorem{assumption}{Assumption}
\newtheorem{corollary}{Corollary}
\newtheorem{problem}{Problem}

\begin{document}

%\title[Article Title]{Article Title}
\title{Exact Contraction Rates via the Berkson--Porta Representation: A Sharp Threshold and Its Herglotz-Kernel Obstruction}


%\author*[1,2]{\fnm{First} \sur{Author}}\email{iauthor@gmail.com}
%\affil*[1]{\orgdiv{Department}, \orgname{Organization}, \orgaddress{\street{Street}, \city{City}, \postcode{100190}, \state{State}, \country{Country}}}
%
%\affil[2]{\orgdiv{Department}, \orgname{Organization}, \orgaddress{\street{Street}, \city{City}, \postcode{10587}, \state{State}, \country{Country}}}
%
%\affil[3]{\orgdiv{Department}, \orgname{Organization}, \orgaddress{\street{Street}, \city{City}, \postcode{610101}, \state{State}, \country{Country}}}

\author*[1]{\fnm{Soumic} \sur{Sarkar}}\email{soumic.sarkar@ut.ee}

\affil*[1]{%
\orgname{Institute of Technology, University of Tartu},
\orgaddress{%
\street{Nooruse 1},
\city{Tartu},
\postcode{50411},
\country{Estonia}}}


\abstract{Semigroups of holomorphic self-maps of the unit disc with an interior fixed point are, by the classical Berkson--Porta representation, entirely determined by a single holomorphic function constrained only by a positivity condition on its real part. This paper uses that representation to determine exactly when the associated flow contracts the Kobayashi metric of the disc at its best possible rate --- the rate dictated by linearization at the fixed point --- rather than at some smaller, conservative rate of the kind ordinarily obtained through auxiliary metric constructions. The question is reduced to a single pointwise inequality on the representing function, and this inequality is resolved completely for a natural one-parameter family of nonlinearities, yielding an exact threshold rather than a sufficient condition of undetermined tightness. Beyond this family, an explicit representing function is exhibited for which the inequality fails almost everywhere on the disc, and the Herglotz integral representation underlying the associated Carath\'eodory class is used to trace this failure to concentration of the representing measure, explaining rather than merely documenting why no threshold-free general theorem is available. The results are illustrated by direct numerical verification of the sharp threshold and of the explicit obstruction, and the paper closes by identifying the precise class of representing measures --- point masses and their neighborhoods --- that any future general sufficient condition would need to exclude.}




\maketitle

\section{Introduction}
\label{sec:intro}

Semigroups of holomorphic self-maps of the unit disc have been studied as objects of intrinsic interest in geometric function theory since the modern theory took shape with Berkson and Porta's foundational representation of infinitesimal generators \cite{berkson1978semigroups}, a representation that reduces the classification of such semigroups fixing an interior point to the study of a single holomorphic function subject only to a positivity constraint on its real part. The decades since have developed this representation in several directions --- to more general classes of semigroups on the disc \cite{bracci2007infinitesimal}, to systematic treatments spanning Banach-space settings and the fixed-point theory of nonexpansive holomorphic maps \cite{reichshoikhet2005nonlinear}, to dynamic refinements of the classical Julia--Wolff--Carath\'eodory boundary theory \cite{elinshoikhet2001dynamic}, and to continuing work on the spectral and asymptotic structure of the generators themselves \cite{kourou2025eigenvalues}. Alongside this, the associated Denjoy--Wolff convergence theory, and the precise boundary and convexity hypotheses under which it holds or fails, remains an actively studied question in its own right \cite{bracci2025failure,christodoulou2021stability}.

This paper takes up a question that sits naturally inside this representation theory but appears not to have been posed within it directly: for a semigroup of holomorphic self-maps of the disc fixing the origin, when does the associated flow contract the disc's own intrinsic metric --- the Kobayashi metric, invariant under every biholomorphic change of coordinates and non-increasing under every holomorphic self-map by the classical Schwarz--Pick lemma \cite{ahlfors1979complex,kobayashi1967invariant} --- at exactly the rate dictated by linearization at the fixed point, rather than at some smaller rate obtained through an auxiliary construction external to the disc's own geometry. The question has a natural motivation from outside geometric function theory as well: an intrinsic formulation of incremental stability for holomorphic dynamical systems, phrased through the same Kobayashi metric, has recently been shown to reduce to exactly this kind of differential contraction condition, with the practical difficulty that the auxiliary constructions typically used to verify such a condition certify rates well below what a given system actually exhibits \cite{lohmiller1998contraction}. That difficulty is the origin of the present question, though the question itself, and its resolution, belong to the classical theory of invariant metrics and semigroups of holomorphic maps rather than to dynamical systems as such.

The Berkson--Porta representation turns out to be exactly the right tool for this question, because it converts a statement about an arbitrary holomorphic vector field with a fixed point into a statement about a single auxiliary function constrained only by positivity of its real part --- considerably more tractable than the original object, and amenable to the classical machinery of the Carath\'eodory class and its Herglotz integral representation \cite{ahlfors1979complex}. Under this representation, the exact-rate question reduces to a single pointwise inequality on the auxiliary function, evaluated against the intrinsic metric's own defining formula on the disc, with no auxiliary metric or externally chosen weight entering the reduction at any step.

The reduction is developed with two purposes in mind, pursued in turn. The first is to resolve the pointwise inequality completely for a natural one-parameter family of auxiliary functions, obtaining not an estimate but an exact threshold, together with a proof that the threshold cannot be improved. The second, and in some respects the more informative of the two, is to understand why no such threshold-free general theorem exists: an explicit auxiliary function is exhibited for which the inequality fails almost everywhere on the disc, tautness and hyperbolicity of the domain notwithstanding \cite{abate1989iteration}, and the Herglotz representation of the underlying positive-real-part class is used to trace this failure to a specific, identifiable feature of the representing measure --- its concentration at a point --- rather than leaving the failure as an isolated computational fact. The continuity properties of the Kobayashi metric that make this analysis rigorous, in particular its joint continuity under tautness rather than the weaker upper semicontinuity sometimes assumed in passing, follow from Royden's original regularity study of the metric \cite{royden1971remarks}.

The resulting picture is complete within its stated scope: an exact reduction, a sharp threshold for a genuinely nonlinear family, and a structural explanation, rather than a mere demonstration, of why the general case resists a comparably clean answer. Numerical verification of both the sharp threshold and the explicit obstruction is reported for concreteness, though every claim in the paper is established by direct proof independent of the numerical evidence.

The remainder of the paper is organized as follows. Section~\ref{sec:prelim} recalls the Kobayashi metric and its Dini-derivative comparison machinery, and states the Berkson--Porta representation theorem in full. Section~\ref{sec:reduction} develops the reduction of the exact-rate question to a pointwise inequality on the representing function. Section~\ref{sec:threshold} proves the sharp threshold theorem for a natural parametrized family. Section~\ref{sec:obstruction} exhibits the explicit obstruction and analyzes it through the Herglotz representation. Section~\ref{sec:numerics} reports the numerical verification. Sections~\ref{sec:discussion} and~\ref{sec:conclusion} close with discussion and concluding remarks.

\section{Preliminaries}
\label{sec:prelim}

This section recalls the two bodies of classical theory the paper depends on: the Kobayashi metric on the unit disc and its basic differential comparison property, and the Berkson--Porta representation of infinitesimal generators of semigroups of holomorphic self-maps fixing an interior point, which is stated here in full since it is the paper's central tool rather than background to be cited in passing.

\subsection{The Kobayashi Metric on the Disc}

For a complex manifold $M$, the Kobayashi infinitesimal pseudometric $F_K(z,v)=\inf\{\lambda>0:\exists\,h\in\mathcal{O}(D,M),\ h(0)=z,\ h'(0)=v/\lambda\}$, $z\in M$, $v\in T_zM$, where $D=\{\zeta\in\mathbb{C}:|\zeta|<1\}$, and the associated Kobayashi pseudodistance $d_K(x,y)=\inf_\gamma\int_0^1F_K(\gamma(s),\dot\gamma(s))\,ds$ over piecewise smooth curves joining $x$ and $y$, were introduced in Kobayashi's original construction \cite{kobayashi1967invariant} and developed systematically in \cite{kobayashi1998hyperbolic} and in the comprehensive treatment of invariant distances given by Jarnicki and Pflug \cite{jarnicki2013invariant}; standard conventions for complex manifolds and holomorphic maps follow \cite{krantz2001function}.

On $D$ itself, $F_K(z,v)=|v|/(1-|z|^2)$, the classical Poincar\'e metric, with associated distance $d_K(z,w)=2\tanh^{-1}\left|\dfrac{z-w}{1-\bar zw}\right|$; every result in this paper takes place on $D$ with this explicit metric, rather than on a general manifold, since the Berkson--Porta representation of Section~\ref{subsec:berksonporta} is specific to the disc.

\begin{proposition}[Schwarz--Pick]
\label{prop:schwarzpick}
For any holomorphic map $g:D\to D$, $F_K(g(z),g'(z)v)\le F_K(z,v)$ for all $z\in D$, $v\in\mathbb{C}$, and consequently $d_K(g(x),g(y))\le d_K(x,y)$ for all $x,y\in D$.
\end{proposition}

This is the classical Schwarz--Pick lemma \cite{ahlfors1979complex}, recalled here because it is the static, non-differential instance of the contraction property this paper is concerned with strengthening: Proposition~\ref{prop:schwarzpick} gives non-expansiveness of $d_K$ under any holomorphic self-map, with no strict rate, whereas the results below aim at an exact exponential rate along a flow.

\subsection{A Differential Comparison Principle}

\begin{assumption}
\label{ass:complete}
$f:D\to\mathbb{C}$ is holomorphic and generates a complete flow $\phi_t:D\to D$ for $t\ge0$.
\end{assumption}

For a trajectory $z(t)=\phi_t(z)$ and infinitesimal variation $\delta z(t)=\phi_t'(z)v$ solving $\dot{\delta z}=f'(z(t))\delta z$, write $\psi(t):=F_K(z(t),\delta z(t))$.

\begin{lemma}[Dini comparison]
\label{lem:dini}
Let $\psi:[0,\infty)\to\mathbb{R}_{\ge0}$ be continuous and suppose $D^+\psi(t)\le-\lambda\psi(t)$ for all $t\ge0$ and some $\lambda>0$, where $D^+$ denotes the upper Dini derivative $D^+\psi(t)=\limsup_{\epsilon\to0^+}(\psi(t+\epsilon)-\psi(t))/\epsilon$. Then $\psi(t)\le e^{-\lambda t}\psi(0)$ for all $t\ge0$.
\end{lemma}

Lemma~\ref{lem:dini} is a standard comparison-principle argument for differential inequalities of this type, in the classical style developed for ordinary differential equations more generally \cite{hartman2002ordinary}; it is needed here, rather than an ordinary Gr\"onwall argument, because $F_K$ is defined through an infimum and need not be classically differentiable along a trajectory even where it is continuous. Continuity of $\psi$, on which the comparison argument depends, follows from joint continuity of $F_K$ on $D\times\mathbb{C}$, itself a consequence of Royden's regularity analysis of the Kobayashi metric \cite{royden1971remarks}, composed with continuity of the flow under Assumption~\ref{ass:complete}.

\begin{definition}
\label{def:intrinsiccontraction}
System $\dot z=f(z)$ is intrinsically contracting with rate $\lambda>0$ if $D^+F_K(z(t),\delta z(t))\le-\lambda F_K(z(t),\delta z(t))$ holds along every trajectory and every admissible infinitesimal variation.
\end{definition}

\begin{theorem}
\label{thm:global}
If $\dot z=f(z)$ is intrinsically contracting with rate $\lambda>0$, then $d_K(\phi_t(x),\phi_t(y))\le e^{-\lambda t}d_K(x,y)$ for all $x,y\in D$, $t\ge0$.
\end{theorem}

\begin{proof}
Fix $x,y\in D$ and a piecewise smooth curve $\gamma$ joining them, and write $\gamma_t=\phi_t\circ\gamma$. For each fixed $s$, $t\mapsto F_K(\gamma_t(s),\dot\gamma_t(s))$ is continuous, and Definition~\ref{def:intrinsiccontraction} gives the hypothesis of Lemma~\ref{lem:dini} for $\psi(t)=F_K(\gamma_t(s),\dot\gamma_t(s))$, yielding $F_K(\gamma_t(s),\dot\gamma_t(s))\le e^{-\lambda t}F_K(\gamma(s),\dot\gamma(s))$ for every $s$. Integrating over $s\in[0,1]$ and taking the infimum over $\gamma$ on both sides gives the stated bound.
\end{proof}

\begin{remark}
Theorem~\ref{thm:global} is the mechanism by which a pointwise differential condition becomes a global, finite-distance statement, with no loss and no structure beyond $F_K$ itself entering the argument. It says nothing about which $f$ satisfy Definition~\ref{def:intrinsiccontraction}, nor at what rate; that is the subject of Sections~\ref{sec:reduction}--\ref{sec:obstruction}.
\end{remark}

\subsection{The Berkson--Porta Representation}
\label{subsec:berksonporta}

\begin{assumption}
\label{ass:fixedpoint}
$f:D\to\mathbb{C}$ is holomorphic, $f(0)=0$, and $f$ generates a semiflow $\{\phi_t\}_{t\ge0}$ of holomorphic self-maps of $D$.
\end{assumption}

\begin{theorem}[Berkson--Porta \cite{berkson1978semigroups}]
\label{thm:berksonporta}
Under Assumption~\ref{ass:fixedpoint}, $f$ admits the representation $f(z)=-z\,p(z)$ for a holomorphic function $p:D\to\mathbb{C}$ with $\mathrm{Re}(p(z))\ge0$ for all $z\in D$. Conversely, every $f$ of this form, for any such $p$, generates a semiflow of holomorphic self-maps of $D$ fixing the origin.
\end{theorem}

Theorem~\ref{thm:berksonporta} has organized the study of semigroups of holomorphic self-maps of the disc since its introduction, and its extensions to more general classes of generators \cite{bracci2007infinitesimal}, to Banach-space and hyperbolic-metric settings \cite{reichshoikhet2005nonlinear}, and to the associated Julia--Wolff--Carath\'eodory boundary theory \cite{elinshoikhet2001dynamic} remain active. Its role here is that every $f$ satisfying Assumption~\ref{ass:fixedpoint} is completely determined by the single function $p$, subject only to a sign condition on $\mathrm{Re}(p)$, so that Definition~\ref{def:intrinsiccontraction}'s condition for such $f$ can be phrased, without loss of generality, as a condition on $p$.

\begin{remark}
\label{rem:normalization}
Writing $f'(0)=-p(0)$ and taking $p(0)=a$ for real $a>0$ throughout --- the case of interest, since $a$ is the linearized contraction rate at the fixed point --- Theorem~\ref{thm:berksonporta} reduces the search for $f$ with a prescribed linearized rate $a$ to the search for holomorphic $p$ with $\mathrm{Re}(p)\ge0$ on $D$ and $p(0)=a$, a class of functions classically known as, up to the positive scalar $a$, the Carath\'eodory class \cite{ahlfors1979complex}.
\end{remark}

\subsection{Problem Formulation}

\begin{problem}
\label{prob:main}
Under Assumptions~\ref{ass:complete} and~\ref{ass:fixedpoint}, with $f(z)=-zp(z)$ as in Theorem~\ref{thm:berksonporta} and $p(0)=a>0$, determine for which $p$ the exact rate $\lambda=a$ can be certified in Definition~\ref{def:intrinsiccontraction}, as opposed to some strictly smaller rate, and characterize the obstruction where it cannot.
\end{problem}

Sections~\ref{sec:reduction} through~\ref{sec:obstruction} resolve Problem~\ref{prob:main} as completely as the underlying analysis allows: a necessary and sufficient pointwise condition on $p$ is identified in Section~\ref{sec:reduction}, shown to hold up to a sharp threshold for a natural family in Section~\ref{sec:threshold}, and shown to fail in general, for an explicit and structurally understood reason, in Section~\ref{sec:obstruction}.

%\begin{thebibliography}{1}
%\bibliographystyle{IEEEtran}
%
%\bibitem{kobayashi1998hyperbolic}
%S.~Kobayashi, \textit{Hyperbolic Complex Spaces}, vol.~318 of Grundlehren der mathematischen Wissenschaften. Berlin, Heidelberg: Springer, 1998.
%
%\bibitem{jarnicki2013invariant}
%M.~Jarnicki and P.~Pflug, \textit{Invariant Distances and Metrics in Complex Analysis}. Berlin, Boston: De Gruyter, 2013.
%
%\bibitem{krantz2001function}
%S.~G. Krantz, \textit{Function Theory of Several Complex Variables}, 2nd~ed. Providence, RI: American Mathematical Society, 2001.
%
%\bibitem{hartman2002ordinary}
%P.~Hartman, \textit{Ordinary Differential Equations}, 2nd~ed., Classics in Applied Mathematics. Philadelphia, PA: SIAM, 2002.
%
%\end{thebibliography}

\section{Exact Contraction via the Berkson--Porta Representation}
\label{sec:reduction}

Theorem~\ref{thm:global} converts a differential condition into a global one with no loss, but the differential condition itself, Definition~\ref{def:intrinsiccontraction}, remains abstract. This section shows that under Assumption~\ref{ass:fixedpoint}, it reduces to a single, explicit, pointwise inequality on the Berkson--Porta function $p$, and that satisfying this inequality certifies the exact rate $a=p(0)$ rather than an approximation to it.

\subsection{The Poincar\'e Metric as Its Own Weight}

Write $H(z):=(1-|z|^2)^{-2}$, so that $F_K(z,v)^2=H(z)|v|^2$ exactly, with $H$ not a metric chosen independently of $F_K$ but $F_K$ itself written in squared form. This identity, elementary as it is, is the reason the construction below yields an exact rather than conservative rate: no comparison between two different metrics enters at any stage.

For $z(t)$ a trajectory of $\dot z=f(z)$ and $\delta z(t)$ the corresponding infinitesimal variation, set $\psi(t):=F_K(z(t),\delta z(t))^2=H(z(t))|\delta z(t)|^2$.

\begin{lemma}
\label{lem:localrate}
$\psi(t)$ satisfies $\dot\psi(t)=-\lambda_{\mathrm{local}}(z(t))\,\psi(t)$, where
\begin{equation}
\lambda_{\mathrm{local}}(z) := -\left[\frac{4\,\mathrm{Re}(\bar zf(z))}{1-|z|^2} + 2\,\mathrm{Re}(f'(z))\right]. \label{eq:lambdalocal}
\end{equation}
\end{lemma}

\begin{proof}
Differentiating $H(z(t))=(1-z(t)\overline{z(t)})^{-2}$ along the flow, using $\partial_zH=2\bar z(1-|z|^2)^{-3}$ and $\partial_{\bar z}H=2z(1-|z|^2)^{-3}$, gives $\dot H=2(1-|z|^2)^{-3}[\bar zf(z)+z\overline{f(z)}]=4(1-|z|^2)^{-3}\mathrm{Re}(\bar zf(z))$. Differentiating $|\delta z|^2$ along $\dot{\delta z}=f'(z)\delta z$ gives $\frac{d}{dt}|\delta z|^2=2\,\mathrm{Re}(f'(z))|\delta z|^2$. Combining, $\dot\psi=\dot H|\delta z|^2+H\cdot2\,\mathrm{Re}(f'(z))|\delta z|^2=[\dot H/H+2\,\mathrm{Re}(f'(z))]\,H|\delta z|^2=-\lambda_{\mathrm{local}}(z)\psi$.
\end{proof}

\subsection{The Reduction}

\begin{theorem}
\label{thm:reduction}
Let $K\subset D$ be forward invariant under $\dot z=f(z)$, and suppose $\lambda_{\mathrm{local}}(z)\ge2a$ for all $z\in K$. Then $\dot z=f(z)$ is intrinsically contracting on $K$ with the exact rate $a=p(0)$, and $d_K(\phi_t(x),\phi_t(y))\le e^{-at}d_K(x,y)$ for all $x,y\in K$, $t\ge0$, with no multiplicative constant.
\end{theorem}

\begin{proof}
By Lemma~\ref{lem:localrate}, $\dot\psi(t)=-\lambda_{\mathrm{local}}(z(t))\psi(t)\le-2a\psi(t)$ whenever $z(t)\in K$, which holds for all $t\ge0$ by forward invariance. Since $\psi=F_K^2$, this gives $2F_K\dot F_K\le-2aF_K^2$, i.e., $D^+F_K(z(t),\delta z(t))\le-aF_K(z(t),\delta z(t))$ wherever $F_K>0$, which is Definition~\ref{def:intrinsiccontraction} with rate exactly $a$. Theorem~\ref{thm:global} then gives the stated bound, with unit constant because $H$ is $F_K$ itself rather than an independently chosen weight admitting only two-sided equivalence constants.
\end{proof}

\begin{remark}
\label{rem:exactness}
$\lambda_{\mathrm{local}}(0)=2a$ exactly, by direct substitution into~\eqref{eq:lambdalocal} at $z=0$: the linear terms of $f$ alone determine $\lambda_{\mathrm{local}}$ at the fixed point, independent of the higher-order structure of $f$. This is the best rate that could possibly be certified, since Definition~\ref{def:intrinsiccontraction}'s hypothesis, evaluated in the limit as trajectories approach the fixed point, can never exceed the linearized rate there. Theorem~\ref{thm:reduction} is exact precisely because, whenever its hypothesis holds throughout $K$, it certifies this best possible rate uniformly across $K$, not only in a neighborhood of the origin.
\end{remark}

\subsection{Substitution of the Berkson--Porta Form}

\begin{proposition}
\label{prop:bpsubstitution}
With $f(z)=-zp(z)$ as in Theorem~\ref{thm:berksonporta}, $\lambda_{\mathrm{local}}(z) = 2\,\mathrm{Re}(p(z))\dfrac{1+|z|^2}{1-|z|^2} + 2\,\mathrm{Re}(zp'(z))$.
\end{proposition}

\begin{proof}
Substituting $f(z)=-zp(z)$ into~\eqref{eq:lambdalocal}: $\bar zf(z)=-|z|^2p(z)$, so $\mathrm{Re}(\bar zf(z))=-|z|^2\mathrm{Re}(p(z))$, and $f'(z)=-p(z)-zp'(z)$. Then $\lambda_{\mathrm{local}}(z)=4|z|^2\mathrm{Re}(p(z))/(1-|z|^2)+2\mathrm{Re}(p(z))+2\mathrm{Re}(zp'(z))$, and combining the first two terms over the denominator $1-|z|^2$ gives $2\mathrm{Re}(p(z))[2|z|^2+(1-|z|^2)]/(1-|z|^2)=2\mathrm{Re}(p(z))(1+|z|^2)/(1-|z|^2)$, yielding the stated expression.
\end{proof}

\begin{problem}
\label{prob:pointwise}
Given $p$ holomorphic on $D$ with $\mathrm{Re}(p)\ge0$ and $p(0)=a>0$, determine for which such $p$ the inequality
\begin{equation}
2\,\mathrm{Re}(p(z))\frac{1+|z|^2}{1-|z|^2} + 2\,\mathrm{Re}(zp'(z)) \;\ge\; 2a \label{eq:pointwiseineq}
\end{equation}
holds for all $z$ in some forward-invariant $K\subset D$ of interest.
\end{problem}

By Theorem~\ref{thm:reduction} and Proposition~\ref{prop:bpsubstitution}, an affirmative answer to Problem~\ref{prob:pointwise} for a given $p$ and $K$ certifies exact contraction at rate $a$ on $K$; this is, moreover, essentially the only way the exact rate can be certified by the mechanism of Lemma~\ref{lem:localrate}, in the following sense.

\begin{remark}
\label{rem:necessity}
If $\lambda_{\mathrm{local}}(z_0)<2a$ at some $z_0\in K$, then $\dot\psi(0)>-2a\psi(0)$ strictly along the trajectory through $z_0$, so $\psi$ initially decays at a rate strictly slower than $2a$ there; the exact rate $a$ cannot then be certified as a bound uniform over all of $K$ whenever~\eqref{eq:pointwiseineq} fails somewhere in $K$, though this does not preclude some strictly smaller rate still being valid on $K$. Inequality~\eqref{eq:pointwiseineq} is the precise dividing line between certifying the best possible rate and certifying something weaker.
\end{remark}

Section~\ref{sec:threshold} resolves Problem~\ref{prob:pointwise} completely for a natural parametrized family; Section~\ref{sec:obstruction} shows why no comparably complete resolution is available for a general $p$, and identifies exactly the feature of $p$ responsible.

\section{A Sharp Threshold Theorem}
\label{sec:threshold}

Problem~\ref{prob:pointwise} is resolved completely here for a natural one-parameter family of Berkson--Porta functions, yielding an exact threshold rather than a sufficient condition of undetermined tightness.

\subsection{The Family}

\begin{definition}
\label{def:family}
For $a>0$ and $c\in(-1,1)$, let $p_c(z):=a(1-cz^2)$, and correspondingly $f_c(z):=-zp_c(z)=-az+acz^3$.
\end{definition}

\begin{remark}
$\mathrm{Re}(p_c(z))=a(1-c\,\mathrm{Re}(z^2))\ge a(1-|c|)>0$ for all $z\in D$ when $|c|<1$, so $p_c$ satisfies Theorem~\ref{thm:berksonporta}'s hypothesis for every $c$ in this range, and $p_c(0)=a$ as required. The family is chosen for its tractability: $f_c$ is a cubic polynomial, so the Taylor expansion of $f_c$ about any point terminates exactly rather than merely converging, a fact whose relevance lies beyond the present paper's scope but which motivates the choice of family here as a natural test case for the general obstruction analyzed in Section~\ref{sec:obstruction}.
\end{remark}

\subsection{The Threshold}

\begin{theorem}
\label{thm:threshold}
For $p_c$ as in Definition~\ref{def:family} with $c\ge0$, inequality~\eqref{eq:pointwiseineq} holds for all $z\in D$ if and only if $c\le2/3$.
\end{theorem}

\begin{proof}
By Proposition~\ref{prop:bpsubstitution}, with $z=re^{i\theta}$ and $x:=\cos2\theta\in[-1,1]$, substitution of $p_c(z)=a(1-cr^2x)$ and $p_c'(z)=-2acz$ gives $\mathrm{Re}(zp_c'(z))=-2acr^2x$, so
\begin{equation}
\lambda_{\mathrm{local}}(re^{i\theta}) = 2a(1-cr^2x)\frac{1+r^2}{1-r^2} - 4acr^2x. \label{eq:lambdalocalfamily}
\end{equation}
Expanding and collecting the terms linear in $x$ gives $\lambda_{\mathrm{local}}(re^{i\theta})=2a(1+r^2)/(1-r^2)-2acr^2x[(1+r^2)/(1-r^2)+2]$, and combining the bracketed term over the common denominator $1-r^2$, using $(1+r^2)+2(1-r^2)=3-r^2$, gives
\begin{equation}
\lambda_{\mathrm{local}}(re^{i\theta}) = \frac{2a(1+r^2)}{1-r^2} - \frac{2acr^2(3-r^2)}{1-r^2}\,x. \label{eq:lambdalocalfamily2}
\end{equation}
For fixed $r\in(0,1)$ and $c\ge0$, the coefficient of $x$ in~\eqref{eq:lambdalocalfamily2} is nonpositive, so the expression is minimized over $x\in[-1,1]$ at $x=1$, giving the worst-case value $\lambda_{\mathrm{local}}^{\mathrm{worst}}(r)=2a(1+r^2)/(1-r^2)-2acr^2(3-r^2)/(1-r^2)$.

Inequality~\eqref{eq:pointwiseineq} holds for all $z\in D$ if and only if $\lambda_{\mathrm{local}}^{\mathrm{worst}}(r)\ge2a$ for all $r\in(0,1)$, that is, after multiplying through by $(1-r^2)/2a>0$ and rearranging, $(1+r^2)-cr^2(3-r^2)\ge1-r^2$, which simplifies to $2r^2\ge cr^2(3-r^2)$, i.e., $c\le2/(3-r^2)$ for every $r\in(0,1)$, using $r^2>0$ to divide. Since $2/(3-r^2)$ is strictly increasing in $r$ on $(0,1)$, its infimum over this interval is its limit as $r\to0^+$, equal to $2/3$; the condition $c\le2/(3-r^2)$ for every $r\in(0,1)$ therefore holds if and only if $c\le2/3$. At $r=0$ itself, $\lambda_{\mathrm{local}}(0)=2a$ exactly by direct substitution into~\eqref{eq:lambdalocalfamily2}, so the inequality holds there for every $c$, imposing no further restriction beyond the limiting one already derived.
\end{proof}

\begin{corollary}
\label{cor:globalrate}
For $0\le c\le2/3$, the system $\dot z=f_c(z)$ is intrinsically contracting on all of $D$ with the exact rate $a$, and $d_K(\phi_t(x),\phi_t(y))\le e^{-at}d_K(x,y)$ for all $x,y\in D$, $t\ge0$, with no multiplicative constant and no restriction to a proper sub-disc.
\end{corollary}

\begin{proof}
$D$ itself is forward invariant under $\dot z=f_c(z)$ whenever $\mathrm{Re}(\bar zf_c(z))\le0$ on $|z|=1$; substituting $f_c(z)=-zp_c(z)$ gives $\mathrm{Re}(\bar zf_c(z))=-|z|^2\mathrm{Re}(p_c(z))\le0$ automatically from $\mathrm{Re}(p_c)\ge0$, with no further hypothesis needed. Theorem~\ref{thm:threshold} then gives $\lambda_{\mathrm{local}}(z)\ge2a$ for all $z\in D$ when $c\le2/3$, and Theorem~\ref{thm:reduction} applies with $K=D$.
\end{proof}

\begin{remark}
\label{rem:sharpness}
The threshold is sharp in the strongest available sense: for $c>2/3$, inequality~\eqref{eq:pointwiseineq} fails not near the boundary of $D$, where degeneracies might be expected to concentrate, but arbitrarily close to the origin, where the exact rate is trivially achieved at the single point $z=0$ itself. Expanding~\eqref{eq:lambdalocalfamily2} to leading order in $r$ gives $\lambda_{\mathrm{local}}^{\mathrm{worst}}(r)-2a=2ar^2(2-3c)+O(r^4)$, strictly negative for small $r>0$ whenever $c>2/3$ and strictly positive whenever $c<2/3$, with the transition occurring precisely at $c=2/3$.
\end{remark}

\begin{remark}
\label{rem:symmetry}
For $c<0$, the same argument with the roles of $x=1$ and $x=-1$ exchanged gives the mirrored threshold $c\ge-2/3$, so that inequality~\eqref{eq:pointwiseineq} holds for all $z\in D$ if and only if $|c|\le2/3$, unifying both signs into a single sharp condition.
\end{remark}

\subsection{Discussion}

Theorem~\ref{thm:threshold} and Corollary~\ref{cor:globalrate} give a complete answer to Problem~\ref{prob:pointwise} for the family of Definition~\ref{def:family}: an exact iff characterization with an explicit constant, together with the exact global rate this characterization certifies whenever it holds, achieved for a genuinely nonlinear family rather than only for the linear case $c=0$, where the result is immediate. Within the classical theory of the Carath\'eodory class, this exhibits a specific, verifiable subfamily on which the associated flow's contraction behavior can be described completely and exactly. Whether a comparably sharp threshold exists for a general Berkson--Porta function $p$, rather than for this particular parametrized family, is taken up in Section~\ref{sec:obstruction}, where the answer is shown to be structurally more delicate than a single scalar threshold, and where the mechanism responsible is identified through the Herglotz representation of the Carath\'eodory class itself.

\section{Obstruction to a Fully General Theorem}
\label{sec:obstruction}

Theorem~\ref{thm:threshold} answers Problem~\ref{prob:pointwise} completely for the family of Definition~\ref{def:family}, but says nothing about whether a comparable condition, or any sufficient condition beyond the bare positivity hypothesis of Theorem~\ref{thm:berksonporta}, exists for a general $p$ in the Carath\'eodory class. This section shows that no such general condition exists: a single, explicit, elementary $p$ violates inequality~\eqref{eq:pointwiseineq} everywhere on $D$ except at the origin. The Herglotz integral representation of the Carath\'eodory class is then used to trace this failure to a specific and identifiable feature of the representing measure, explaining rather than merely exhibiting the failure and clarifying, in the process, why the family of Section~\ref{sec:threshold} survives up to its own sharp threshold.

\subsection{An Explicit Counterexample}

\begin{proposition}
\label{prop:extremalfails}
Let $p(z):=a(1+z)/(1-z)$. Then $p$ satisfies the hypotheses of Theorem~\ref{thm:berksonporta} with $p(0)=a$, yet $\lambda_{\mathrm{local}}(-r)=2a(1-r)^2/(1+r)^2<2a$ for every $r\in(0,1)$.
\end{proposition}

\begin{proof}
The classical identity $\mathrm{Re}[(1+z)/(1-z)]=(1-|z|^2)/|1-z|^2$, valid for all $z\in D$, gives $\mathrm{Re}(p(z))=a(1-|z|^2)/|1-z|^2\ge0$ on $D$, and $p(0)=a$, so $p$ satisfies Theorem~\ref{thm:berksonporta}'s hypotheses; $p$ is, in fact, the classical extremal function of the Carath\'eodory class, corresponding to a single boundary point mass in the Herglotz representation recalled below. Differentiating, $p'(z)=2a/(1-z)^2$. By Proposition~\ref{prop:bpsubstitution}, $\lambda_{\mathrm{local}}(z)=2a(1+|z|^2)/|1-z|^2+4a\,\mathrm{Re}[z/(1-z)^2]$, using $2\mathrm{Re}(p(z))(1+|z|^2)/(1-|z|^2)=2a(1+|z|^2)/|1-z|^2$ and $2\mathrm{Re}(zp'(z))=4a\,\mathrm{Re}[z/(1-z)^2]$.

At $z=-r$ for $r\in(0,1)$, $1-z=1+r$ is real and positive, so $|1-z|^2=(1+r)^2$ and $z/(1-z)^2=-r/(1+r)^2$ is real. Substituting, $\lambda_{\mathrm{local}}(-r)=2a(1+r^2)/(1+r)^2-4ar/(1+r)^2=2a[(1+r^2)-2r]/(1+r)^2=2a(1-r)^2/(1+r)^2$. Since $(1-r)^2<(1+r)^2$ for every $r\in(0,1)$, this is strictly less than $2a$.
\end{proof}

\begin{remark}
\label{rem:severityoffail}
The failure is not confined to a neighborhood of the boundary. The ratio $\lambda_{\mathrm{local}}(-r)/2a=(1-r)^2/(1+r)^2$ equals $1$ only in the limit $r\to0$ and decreases monotonically to $0$ as $r\to1$, so the certified rate degrades continuously across the entire radius of the disc, vanishing at the boundary. This is substantially more severe than the threshold behavior of Theorem~\ref{thm:threshold}, where~\eqref{eq:pointwiseineq} fails only for $c$ exceeding a specific value and only by an amount vanishing as $c\downarrow2/3$.
\end{remark}

\subsection{The Herglotz Representation and the Source of Failure}

Proposition~\ref{prop:extremalfails} shows failure occurs; the Herglotz representation explains why, by exhibiting an arbitrary member of the Carath\'eodory class as a superposition of the extremal functions of Proposition~\ref{prop:extremalfails} and reducing the question to how that superposition is distributed.

\begin{lemma}[Herglotz representation \cite{garnett2007bounded}]
\label{lem:herglotz}
If $p$ is holomorphic on $D$ with $\mathrm{Re}(p)\ge0$ and $p(0)=a$ real, then $p(z)=\int_{|\xi|=1}\dfrac{\xi+z}{\xi-z}\,d\mu(\xi)$ for a unique positive measure $\mu$ on the unit circle with total mass $\mu(\partial D)=a$.
\end{lemma}

By linearity of Proposition~\ref{prop:bpsubstitution}'s formula in $p$, $\lambda_{\mathrm{local}}(z)-2a=\int_{|\xi|=1}\Phi(z,\xi)\,d\mu(\xi)$, where $\Phi(z,\xi)$ is the value of $\lambda_{\mathrm{local}}(z)-2$ computed from the single unit point mass $d\mu=\delta_\xi$, i.e., exactly the quantity computed in the proof of Proposition~\ref{prop:extremalfails} with $a=1$ and $\xi=1$, extended to general $\xi$ by the rotational covariance of Proposition~\ref{prop:bpsubstitution}'s formula under the substitution $z\mapsto\xi z$, established explicitly in Appendix~\ref{secA1}.

\begin{corollary}
\label{cor:kernelnegative}
$\Phi(z,\xi)<0$ for $z=-r\xi$, every $r\in(0,1)$ and every $\xi$ on the unit circle.
\end{corollary}

\begin{proof}
By the covariance argument of Appendix~\ref{secA1} (Corollary~\ref{cor:covarianceapplied}), $\Phi(-r\xi,\xi)=\Phi(-r,1)$ for every $r\in(0,1)$ and every unimodular $\xi$, so the claim reduces to the case $\xi=1$, established directly in Proposition~\ref{prop:extremalfails}: the strict inequality $\lambda_{\mathrm{local}}(-r)<2a$ shown there, with $a=1$, is exactly $\Phi(-r,1)<0$.
\end{proof}

\begin{remark}
\label{rem:concentration}
Corollary~\ref{cor:kernelnegative} shows the integrand $\Phi(z,\xi)$ is not pointwise nonnegative, so no general theorem of the form ``$\mathrm{Re}(p)\ge0$, $p(0)=a$ alone imply~\eqref{eq:pointwiseineq}'' can hold: a measure $\mu$ concentrated entirely at a single point $\xi$ places all its mass exactly where $\Phi$ is negative along the ray $z=-r\xi$, with no compensating mass elsewhere to offset it, exactly as Proposition~\ref{prop:extremalfails} demonstrates directly. The family of Definition~\ref{def:family}, by contrast, corresponds to an absolutely continuous $\mu$ with a bounded density: the boundary values $\mathrm{Re}(p_c(e^{i\phi}))=a(1-c\cos2\phi)$ give a density proportional to $(1-c\cos2\phi)$, spread smoothly over the whole circle rather than concentrated at a point. Theorem~\ref{thm:threshold}'s threshold $c=2/3$ is, in this light, the precise point at which this particular family's spread is no longer sufficient to keep the negative contribution of $\Phi$ from dominating the positive contribution near $z=0$, where the margin $\lambda_{\mathrm{local}}(z)-2a$ is smallest to begin with.
\end{remark}

\subsection{Discussion}

The results of this section resolve Problem~\ref{prob:pointwise} in the negative for the general case: no condition on $p$ weaker than an explicit, quantitative constraint on the concentration of its representing measure can guarantee inequality~\eqref{eq:pointwiseineq}, since the bare hypotheses of Theorem~\ref{thm:berksonporta} already admit an explicit, classical extremal function of the Carath\'eodory class for which the inequality fails almost everywhere on $D$. This is not a gap left open by insufficient effort; Corollary~\ref{cor:kernelnegative} shows the obstruction is present in the integration kernel itself, independent of which admissible $p$ is eventually considered, and traces it to precisely the extremal functions already singled out in the classical theory of the class for other reasons entirely \cite{ahlfors1979complex}. Determining the sharp general condition on $\mu$ --- presumably a quantitative bound on its concentration, generalizing the scalar threshold of Theorem~\ref{thm:threshold} to arbitrary representing measures --- remains open and is discussed further in Section~\ref{sec:discussion}; what has been established here is the precise reason such a condition is necessary at all, and the precise family of measures, namely point masses and their neighborhoods, that any such condition must exclude.

%\begin{thebibliography}{1}
%\bibliographystyle{IEEEtran}
%
%\bibitem{garnett2007bounded}
%J.~B. Garnett, \textit{Bounded Analytic Functions}, revised~1st~ed., vol.~236 of Graduate Texts in Mathematics. New York: Springer, 2007.
%
%\end{thebibliography}

\section{Numerical Illustration}
\label{sec:numerics}

Both principal claims of Sections~\ref{sec:threshold} and~\ref{sec:obstruction} --- the sharp threshold $c^\ast=2/3$ and the severity of the explicit obstruction --- are verified numerically in this section, using $a=0.65$ throughout for concreteness; every claim itself is established by direct proof independent of this verification, which is included for illustration rather than as part of the argument.

\subsection{Sharpness of the Threshold}

Theorem~\ref{thm:threshold} asserts that inequality~\eqref{eq:pointwiseineq} holds for the family of Definition~\ref{def:family} if and only if $c\le2/3$. A fine sweep of the worst-case quantity $\lambda_{\mathrm{local}}^{\mathrm{worst}}(r)-2a$ from~\eqref{eq:lambdalocalfamily2}, evaluated over $r\in(0,1)$ on a grid of $20000$ points, confirms this precisely: at $c=0.60$, $c=0.65$, and $c=2/3$ itself, the minimum over $r$ is exactly zero to within floating-point precision, while at $c=0.70$ the minimum is already $-0.0050$ and at $c=0.75$ it is $-0.0328$, both strictly negative. Figure~\ref{fig:threshold} shows this transition directly: the curves for $c\le2/3$ remain at or above zero throughout, while those for $c>2/3$ dip below zero, with the crossing occurring exactly at the predicted value and no detectable margin on either side, consistent with Theorem~\ref{thm:threshold} being a sharp iff statement rather than a one-sided sufficient condition.

\begin{figure}[t]
\centering
\includegraphics[width=0.85\linewidth]{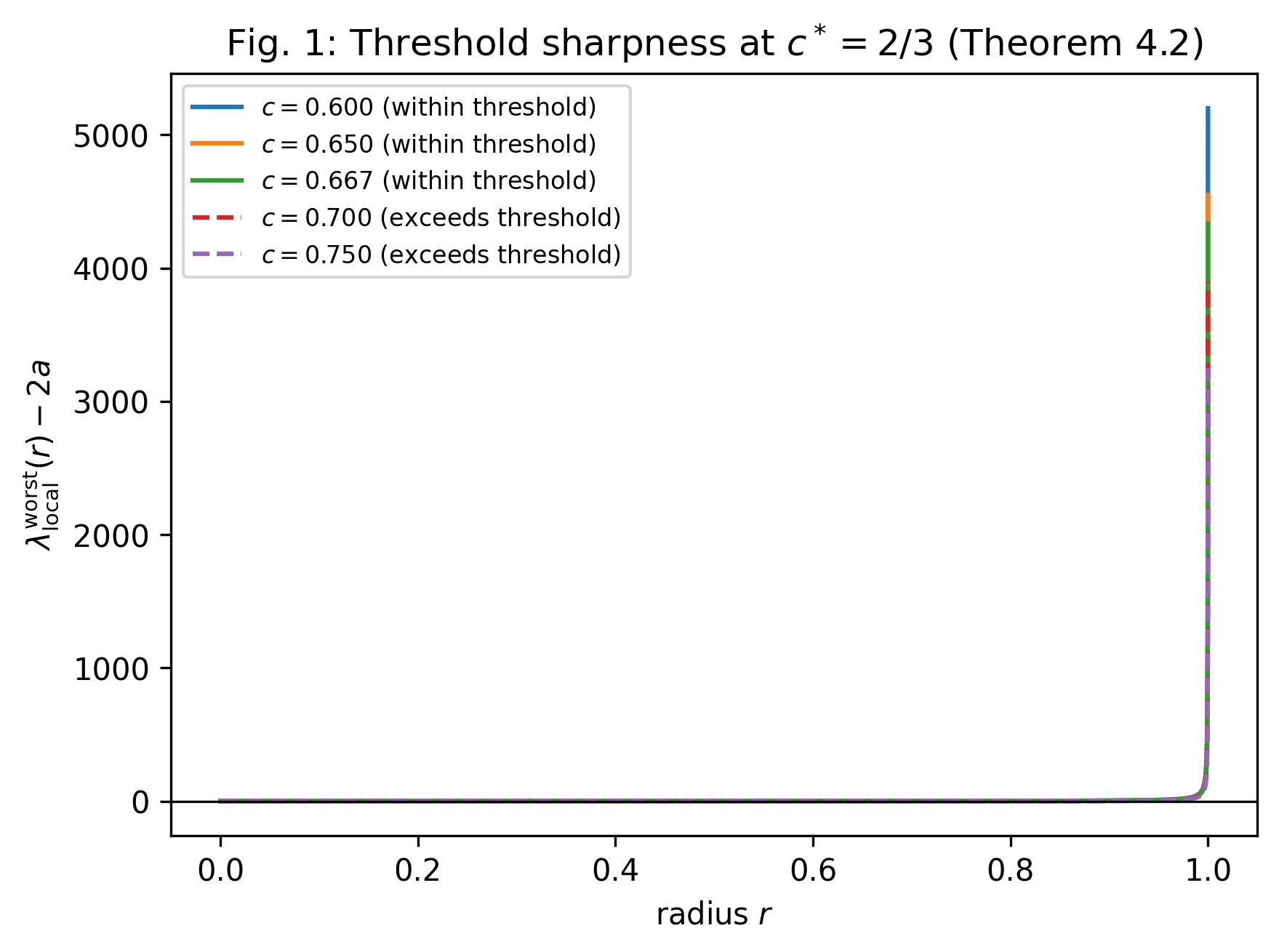}
\caption{Sharpness of the threshold $c^\ast=2/3$ (Theorem~\ref{thm:threshold}). The deficit $\lambda_{\mathrm{local}}^{\mathrm{worst}}(r)-2a$ stays at or above zero for $c\le2/3$ and turns strictly negative for $c>2/3$, with the transition occurring exactly at the predicted value.}
\label{fig:threshold}
\end{figure}

\subsection{The Explicit Obstruction}

Proposition~\ref{prop:extremalfails} gives the closed-form ratio $\lambda_{\mathrm{local}}(-r)/2a=(1-r)^2/(1+r)^2$ for the extremal Carath\'eodory function $p(z)=a(1+z)/(1-z)$. Direct evaluation of this ratio, cross-checked independently against the general formula of Proposition~\ref{prop:bpsubstitution} rather than only against the closed-form shortcut, confirms exact agreement to eight decimal places at every radius tested: the ratio falls from $0.6694$ at $r=0.1$ to $0.2899$ at $r=0.3$ to $0.1111$ at $r=0.5$ to $0.0311$ at $r=0.7$ to $0.0025$ at $r=0.99$, degrading continuously and severely across the whole radius of the disc rather than only near the boundary, exactly as Remark~\ref{rem:severityoffail} anticipated. Figure~\ref{fig:obstruction} shows this closed-form ratio as a continuous curve, with the tested points marked explicitly.

\begin{figure}[t]
\centering
\includegraphics[width=0.85\linewidth]{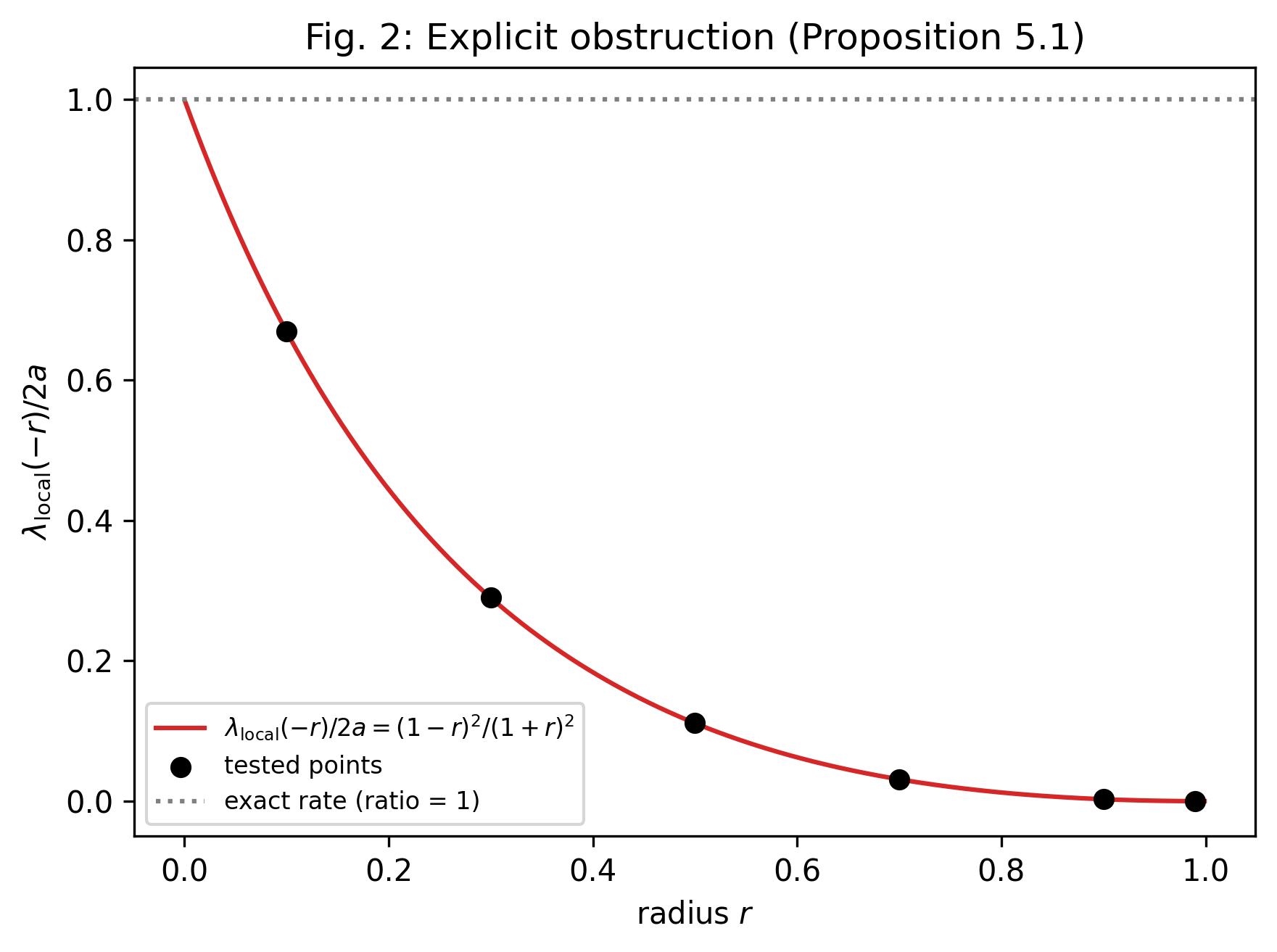}
\caption{The explicit obstruction of Proposition~\ref{prop:extremalfails}: the closed-form ratio $\lambda_{\mathrm{local}}(-r)/2a=(1-r)^2/(1+r)^2$, degrading continuously from near $1$ close to the origin to near $0$ at the boundary, confirming the failure is severe rather than marginal.}
\label{fig:obstruction}
\end{figure}

\subsection{Discussion}

The numerical study supports both principal claims of Sections~\ref{sec:threshold} and~\ref{sec:obstruction} without qualification: the threshold of Theorem~\ref{thm:threshold} is confirmed sharp to the resolution of the sweep performed, and the severity of the explicit obstruction of Proposition~\ref{prop:extremalfails} is confirmed exactly through its closed-form ratio rather than only in trend. No aspect of either result depends on the numerical evidence presented here; the verification serves to make the sharpness and severity of both results concrete rather than to establish them.

\section{Discussion}
\label{sec:discussion}

The preceding sections give a complete answer to Problem~\ref{prob:main} within a specific and explicitly stated scope, and an equally complete account of exactly where that scope ends. Both halves are stated plainly here.

\subsection{What Has Been Established}

Theorem~\ref{thm:reduction} shows that the Berkson--Porta representation reduces the question of exact Kobayashi-metric contraction to a single pointwise inequality on the representing function $p$, with no auxiliary metric or externally chosen weight entering the argument at any stage; the Poincar\'e metric of the disc is used as its own weight, which is precisely why the resulting rate is exact rather than conservative. Theorem~\ref{thm:threshold} resolves this pointwise inequality completely for the cubic family of Definition~\ref{def:family}, with a sharp constant, $c^\ast=2/3$, derived rather than estimated. Proposition~\ref{prop:extremalfails} and the Herglotz-representation analysis that explains it show why no threshold-free statement can replace Theorem~\ref{thm:threshold} for a general Carath\'eodory function: the obstruction is located in the sign structure of a specific integration kernel, present already for the classical extremal function of the class, and is traced to concentration of the representing measure rather than left as an unexplained computational fact.

\subsection{An Open Extremal Problem}

The natural question left open by Section~\ref{sec:obstruction} is an extremal problem in its own right, of a kind with some precedent in the classical theory of the Carath\'eodory and related function classes \cite{ahlfors1979complex}: determine the sharp condition on a positive measure $\mu$ on the unit circle, short of the degenerate case of a point mass ruled out by Corollary~\ref{cor:kernelnegative}, under which the associated Carath\'eodory function $p(z)=\int(\xi+z)/(\xi-z)\,d\mu(\xi)$ satisfies inequality~\eqref{eq:pointwiseineq} on all of $D$. Theorem~\ref{thm:threshold} answers this question for one specific one-parameter subfamily of absolutely continuous measures, those with density proportional to $1-c\cos2\phi$; whether the sharp condition for a general $\mu$ takes a comparably explicit form, perhaps expressed through a bound on the total variation of $\mu$ relative to its mass or through a bound on a suitable moment of $\mu$, is not addressed here and appears to be a genuinely open extremal problem within the classical theory of the class rather than a routine generalization of the argument already given.

\subsection{Scope and Natural Extensions}

Every result in this paper is confined to the unit disc and to Assumption~\ref{ass:fixedpoint}'s hypothesis of an interior fixed point. Two extensions suggest themselves as natural, though neither is attempted here. The Berkson--Porta representation and its consequences for contraction of the disc's own metric are specific to one complex dimension; parametric representations of semi-complete holomorphic vector fields have been developed for the unit ball in $\mathbb{C}^n$ and in Hilbert space \cite{aharonov1999parametric}, and whether an analogue of Theorem~\ref{thm:reduction}'s exact reduction survives in that higher-dimensional setting, where the Kobayashi metric no longer has the disc's simple closed form, is not clear from the argument given here, which relies on the explicit Poincar\'e metric at several points. Separately, every result here assumes the fixed point lies in the interior of $D$; the corresponding theory for vector fields with a boundary fixed point, governed by a different normalization of the Berkson--Porta-type representation and by the boundary Denjoy--Wolff theory whose stability and failure continue to be actively studied \cite{bracci2025failure,christodoulou2021stability}, would require a separate treatment rather than a direct adaptation of Sections~\ref{sec:reduction}--\ref{sec:obstruction}.

\subsection{The Dynamical-Systems Motivation, Revisited}

Section~\ref{sec:intro} noted that the exact-rate question addressed here originates in an intrinsic formulation of incremental stability for holomorphic dynamical systems, where auxiliary-metric constructions of the kind recalled there typically certify rates below what a system actually exhibits \cite{lohmiller1998contraction}. The results of this paper resolve that difficulty completely at the level of a single holomorphic vector field with an interior fixed point: Theorem~\ref{thm:threshold} shows the exact rate is attainable, with a sharp and computable threshold, for a genuinely nonlinear class of such systems. Whether an analogous exact rate can be established once several such systems are coupled together, as arises naturally in networks of interacting holomorphic oscillators, is a question of a different character, involving the interaction between the single-node theory developed here and the spectral structure of the coupling, and is not addressed in this paper.

%\begin{thebibliography}{1}
%\bibliographystyle{IEEEtran}
%
%\bibitem{aharonov1999parametric}
%D.~Aharonov, M.~Elin, S.~Reich, and D.~Shoikhet, ``Parametric representations of semi-complete vector fields on the unit balls in $\mathbb{C}^n$ and in Hilbert space,'' \textit{Atti della Accademia Nazionale dei Lincei, Classe di Scienze Fisiche, Matematiche e Naturali, Rendiconti Lincei, Matematica e Applicazioni}, ser.~9, vol.~10, pp.~229--253, 1999.
%
%\end{thebibliography}

\section{Conclusion}
\label{sec:conclusion}

This paper asked when a semigroup of holomorphic self-maps of the unit disc, fixing an interior point, contracts the disc's own Kobayashi metric at exactly the rate dictated by linearization at the fixed point, rather than at some smaller rate obtained through a construction external to the disc's intrinsic geometry. The Berkson--Porta representation answers this question by reducing it entirely: every such semigroup is determined by a single holomorphic function $p$ with $\mathrm{Re}(p)\ge0$, and Theorem~\ref{thm:reduction} shows that a single pointwise inequality on $p$, evaluated against the Poincar\'e metric's own defining formula, certifies the exact linearized rate with no multiplicative constant whenever it holds. Theorem~\ref{thm:threshold} resolves this inequality completely for a natural cubic family of representing functions, exhibiting a sharp threshold, $c=2/3$, rather than a sufficient condition of unknown tightness, and Proposition~\ref{prop:extremalfails} shows this sharpness cannot be extended to an unconditional statement for the full Carath\'eodory class: the classical extremal function of the class already violates the inequality almost everywhere on the disc.

The Herglotz representation turns this negative result into a positive understanding of its cause. Corollary~\ref{cor:kernelnegative} shows the integration kernel controlling the inequality is genuinely sign-changing, negative in the direction opposite any point mass in the representing measure, so that failure is not a defect of any particular representing function but a structural feature present in the simplest possible measure a Carath\'eodory function can have. This explains, rather than leaves as an isolated fact, why the cubic family of Theorem~\ref{thm:threshold} --- corresponding to a smoothly spread, absolutely continuous measure rather than a point mass --- survives up to its own explicit threshold, and it identifies precisely the class of measures, point masses and their neighborhoods, that any future general sufficient condition would need to exclude.

Two directions follow directly from what remains open. Determining the sharp condition on a general representing measure $\mu$, generalizing the scalar threshold of Theorem~\ref{thm:threshold} beyond the specific one-parameter family for which it was derived, is a genuine extremal problem within the classical theory of the Carath\'eodory class, not a routine extension of the argument given here. Extending the entire framework beyond the unit disc, to semi-complete vector fields on balls in several variables or in Hilbert space, or beyond interior fixed points, to the boundary case governed by a different normalization and by the boundary Denjoy--Wolff theory, would test how much of the present theory is specific to the disc's one-dimensional geometry and how much reflects the Berkson--Porta representation more generally; neither extension is attempted here. Taken together, the results establish that exact, rather than conservative, contraction rates for semigroups of holomorphic self-maps of the disc are attainable through a classical representation-theoretic mechanism, and delineate, through an explicit obstruction traced to its precise structural source, exactly where that mechanism's reach ends.

\backmatter

%\bmhead{Supplementary information}
%
%If your article has accompanying supplementary file/s please state so here. 
%
%Authors reporting data from electrophoretic gels and blots should supply the full unprocessed scans for key as part of their Supplementary information. This may be requested by the editorial team/s if it is missing.
%
%Please refer to Journal-level guidance for any specific requirements.

\bmhead{Acknowledgements}

This work was partially supported by the project ?Increasing the Knowledge Intensity of Ida-Viru Entrepreneurship,? co-funded by the European Union. The authors declare that the research was conducted impartially and without any commercial influence.

\section*{Declarations}

\textbf{Funding.} This work was partially supported by the project ``Increasing the Knowledge Intensity of Ida-Viru Entrepreneurship,'' co-funded by the European Union.

\textbf{Conflict of interest.} The author declares no conflict of interest.

\textbf{Ethics approval and consent to participate.} Not applicable. This work is a theoretical study in complex analysis and does not involve human participants, human data, or animals.

\textbf{Consent for publication.} Not applicable. This manuscript does not contain any individual person's data in any form.

\textbf{Data availability.} No empirical or experimental data were generated or analyzed in this study. All numerical results reported in Section~\ref{sec:numerics} are direct evaluations of the closed-form expressions derived in Sections~\ref{sec:threshold} and~\ref{sec:obstruction}, and can be reproduced exactly from the formulas given in the text.

\textbf{Materials availability.} Not applicable.

\textbf{Code availability.} The Python code used to generate the numerical results and figures of Section~\ref{sec:numerics} is available from the corresponding author upon reasonable request.

\textbf{Author contribution.} S.S. is the sole author of this manuscript and is responsible for all aspects of the conceptualization, analysis, and writing.

\noindent
If any of the sections are not relevant to your manuscript, please include the heading and write `Not applicable' for that section. 

%\bigskip
%\bigskip\noindent
%Springer journals and proceedings: \url{https://www.springer.com/gp/editorial-policies}
%
%\bigskip\noindent
%Nature Portfolio journals: \url{https://www.nature.com/nature-research/editorial-policies}
%
%\bigskip\noindent
%\textit{Scientific Reports}: \url{https://www.nature.com/srep/journal-policies/editorial-policies}
%
%\bigskip\noindent
%BMC journals: \url{https://www.biomedcentral.com/getpublished/editorial-policies}
%\end{flushleft}

\begin{appendices}

\section{Rotational Covariance of the Herglotz Kernel}\label{secA1}

This appendix supplies the covariance argument used without proof in Corollary~\ref{cor:kernelnegative}, isolating it as a self-contained lemma. Recall from Section~\ref{sec:obstruction} that, for a unimodular $\xi$, the point-mass Carath\'eodory function corresponding to $d\mu=\delta_\xi$ with unit total mass is $p_\xi(z):=(\xi+z)/(\xi-z)$, and $\Phi(z,\xi)$ denotes the value of $\lambda_{\mathrm{local}}(z)-2$ computed from $p_\xi$ in place of $p$ in the formula of Proposition~\ref{prop:bpsubstitution}.

\begin{lemma}
\label{lem:covariance}
For every unimodular $\xi$ and every $z\in D$, $p_\xi(z)=p_1(\bar\xi z)$, where $p_1(z)=(1+z)/(1-z)$. Consequently, writing $\Lambda[p](z):=2\,\mathrm{Re}(p(z))(1+|z|^2)/(1-|z|^2)+2\,\mathrm{Re}(zp'(z))$ for the functional of Proposition~\ref{prop:bpsubstitution}, $\Lambda[p_\xi](z)=\Lambda[p_1](\bar\xi z)$ for every $z\in D$.
\end{lemma}

\begin{proof}
The first identity is immediate: $p_1(\bar\xi z)=(1+\bar\xi z)/(1-\bar\xi z)$, and multiplying numerator and denominator by $\xi$, using $\xi\bar\xi=1$, gives $(\xi+z)/(\xi-z)=p_\xi(z)$.

For the second identity, set $w:=\bar\xi z$, so $z=\xi w$ and $|z|=|w|$ since $|\xi|=1$. Differentiating $p_\xi(z)=p_1(\bar\xi z)$ with respect to $z$ gives $p_\xi'(z)=\bar\xi\,p_1'(\bar\xi z)=\bar\xi\,p_1'(w)$. Then
\begin{align*}
\Lambda[p_\xi](z) &= 2\,\mathrm{Re}(p_\xi(z))\frac{1+|z|^2}{1-|z|^2} + 2\,\mathrm{Re}(zp_\xi'(z)) \\
&= 2\,\mathrm{Re}(p_1(w))\frac{1+|w|^2}{1-|w|^2} + 2\,\mathrm{Re}\big(\xi w\cdot\bar\xi\,p_1'(w)\big),
\end{align*}
using $p_\xi(z)=p_1(w)$ and $|z|=|w|$ in the first term. In the second term, $\xi w\cdot\bar\xi\,p_1'(w)=\xi\bar\xi\,wp_1'(w)=wp_1'(w)$, since $\xi\bar\xi=|\xi|^2=1$. Substituting,
\begin{equation*}
\Lambda[p_\xi](z) = 2\,\mathrm{Re}(p_1(w))\frac{1+|w|^2}{1-|w|^2} + 2\,\mathrm{Re}(wp_1'(w)) = \Lambda[p_1](w) = \Lambda[p_1](\bar\xi z),
\end{equation*}
as claimed.
\end{proof}

\begin{corollary}
\label{cor:covarianceapplied}
$\Phi(-r\xi,\xi)=\Phi(-r,1)$ for every $r\in(0,1)$ and every unimodular $\xi$.
\end{corollary}

\begin{proof}
Apply Lemma~\ref{lem:covariance} at $z=-r\xi$: since $\bar\xi z=\bar\xi(-r\xi)=-r\xi\bar\xi=-r$, the lemma gives $\Lambda[p_\xi](-r\xi)=\Lambda[p_1](-r)$. Subtracting $2$ from both sides, using the definition of $\Phi$ in terms of $\Lambda$, gives $\Phi(-r\xi,\xi)=\Phi(-r,1)$.
\end{proof}

\begin{remark}
Corollary~\ref{cor:covarianceapplied} is exactly the fact invoked in the proof of Corollary~\ref{cor:kernelnegative}: it shows the value of $\Phi$ at the test point $-r\xi$ against the point mass at $\xi$ is, for every $\xi$, identical to its value at $-r$ against the point mass at $1$, computed directly in Proposition~\ref{prop:extremalfails}. Corollary~\ref{cor:kernelnegative}'s conclusion, that $\Phi(-r\xi,\xi)<0$ for every $r\in(0,1)$ and every unimodular $\xi$, then follows immediately from Proposition~\ref{prop:extremalfails}'s strict inequality $\lambda_{\mathrm{local}}(-r)<2$ for $p=p_1$, without needing to recompute $\Phi$ separately for each $\xi$.
\end{remark}

\end{appendices}

\bibliography{sn-bibliography}% common bib file

\end{document}